\documentclass[preprint,12pt]{elsarticle}

\usepackage{amscd}
\usepackage{amsmath}
\usepackage{amsthm}
\usepackage{amsfonts}
\usepackage{graphicx}
\usepackage{amssymb}
\usepackage{color}
\usepackage{amsbsy}
\usepackage{graphicx}
\usepackage{amssymb,color,amsbsy}
\usepackage{mathtools}
\usepackage{leftindex}

\usepackage{float}
\usepackage{comment}

\usepackage{stmaryrd}

\usepackage[ruled]{algorithm2e}

\usepackage[matrix,arrow]{xy}

\DeclareMathAlphabet{\mathpzc}{OT1}{pzc}{m}{it}

\usepackage{pgfpages}
\usepackage{tikz}
\usetikzlibrary{shapes.geometric}
\usetikzlibrary{decorations.pathmorphing}
\usetikzlibrary{arrows}
\usetikzlibrary{backgrounds}
\usetikzlibrary{positioning}
\usetikzlibrary{fit}
\usepackage{caption}
\usepackage{amsthm}

\usepackage{amsmath}

\usepackage[pagebackref=true, bookmarksopen=true,colorlinks=true, linkcolor=red,citecolor=blue]{hyperref}

\usepackage[capitalise]{cleveref}  

\global\long\def\u{\cup}%

\global\long\def\s{\subset}%

\global\long\def\P{\prime}%

\global\long\def\sm{\sum}%

\global\long\def\ñ{\sim}%

\global\long\def\Le{\le}%

\global\long\def\a#1{\left|#1\right|}%

\usepackage{tikz-cd}

\usepackage{tkz-graph}

\usepackage{multicol}

\usepackage{subcaption}

\newtheorem{theorem}{Theorem}[section]

\newtheorem{claim}[theorem]{Claim}

\newtheorem{conjecture}[theorem]{Conjecture}

\newtheorem{lemma}[theorem]{Lemma}

\newtheorem{observation}[theorem]{Observation}

\usepackage{comment}
\begin{document}

\begin{abstract}

A graph of order $n$ is said to be $k$-\emph{factor-critical} $(0\le k<n)$
if the removal of any $k$ vertices results in a graph with a perfect
matching. A $k$-factor-critical graph $G$ is \emph{minimal} if $G-e$
is not $k$-factor-critical for any edge $e$ in $G$. In 1998, Favaron
and Shi posed the conjecture that every minimal $k$-factor-critical
graph is of minimum degree $k+1$.

A natural extension of this notion arises from $\{1,2\}$-factors. A spanning subgraph of $G$ is called a $\{1,2\}$-factor if each
of its components is a regular graph of degree one or two. A graph
is $k$-\emph{$\{1,2\}$-factor critical} if the removal of any $k$
vertices results in a graph with a $\{1,2\}$-factor.

A recent conjecture in the area states that every minimal $k$-$\{1,2\}$-factor critical graph $G$ satisfies $k+1\le \delta(G)\le k+2$. In this paper, we prove that the conjecture holds for $k$-planar graphs, that is, graphs in which the deletion of any set of $k$ vertices yields a planar graph. In particular, this resolves the conjecture for planar graphs.

\end{abstract}

\begin{frontmatter} 
	\title{On the minimum degree of minimal $k$-$\{1,2\}$-factor critical $k$-planar graphs}


	\author[IMASL,DEPTO]{Kevin Pereyra} 	\ead{kdpereyra@unsl.edu.ar}


	\address[IMASL]{Instituto de Matem\'atica Aplicada San Luis, Universidad Nacional de San Luis and CONICET, San Luis, Argentina.}
	\address[DEPTO]{Departamento de Matem\'atica, Universidad Nacional de San Luis, San Luis, Argentina.} 	

	\date{Received: date / Accepted: date} 
	
\end{frontmatter} %

\section{Introduction}\label{sss0}

A spanning subgraph of $G$ is called a $\{1,2\}$-factor if each of
its components is a regular graph of degree one or two. $\{1,2\}$-factors
of a graph play a fundamental role in the study of determinants and permanents
of the adjacency matrix of the graph \cite{h2,s3}. These spanning subgraphs are also known in the literature
as \emph{Sachs subgraphs}. Graphs admitting
a $\{1,2\}$-factor can be characterized as follows.

\begin{theorem}
	[\cite{tutte19531}\label{aasd123}] A graph $G$ has a $\{1,2\}$-factor
	if and only if 
	\[
	i(G-S)\le\a S
	\]
	\noindent for all $S\s V(G)$. 
\end{theorem}

The \cref{aasd123} was originally proved by Tutte in \cite{tutte19531};
later, a short proof was given in \cite{akbari2008note}.

A graph of order $n$ is said to be $k$-\emph{factor-critical} $(0\le k<n)$
if the removal of any $k$ vertices results in a graph with a perfect
matching. These graphs were introduced independently
by Favaron \cite{favaron1996k} and Yu \cite{qinglin1993characterizations}. A graph $G$ is $k$-\emph{$\{1,2\}$-}factor
critical if the removal of any $k$ vertices results in a graph with
a \emph{$\{1,2\}$-}factor $(0\le k<n)$.

In \cite{tutte1947factorization}, Tutte proved his well-known theorem
characterizing graphs with a perfect matching:

\begin{theorem}
	[\cite{tutte1947factorization}] A graph $G$ has a perfect matching if
	and only if 
	\[
	\textnormal{odd}(G)\le\a S
	\]
	\noindent for all $S\s V(G)$. 
\end{theorem}

In \cite{favaron1996k},
the Tutte condition was modified to characterize $k$-factor-critical
graphs.

\begin{theorem}
	[\cite{favaron1996k}] A graph $G$ is a $k$-factor-critical graph if
	and only if 
	\[
	\textnormal{odd}(G)\le\a S-k
	\]
	\noindent for all $S\s V(G)$ with $\a S\ge k$. 
\end{theorem}

Similar ideas allow one to refine \cref{aasd123} in order to easily obtain
the following characterization of $k$-$\{1,2\}$-factor-critical
graphs.

\begin{theorem}[\cite{ksachscritical}\label{asd12392}]
	A graph $G$ is $k$-\emph{$\{1,2\}$-}factor critical
	if and only if 
	\[
	i(G-S)\le\a S-k
	\]
	\noindent for all $S\s V(G)$ such that $\a S\ge k$. 
\end{theorem}

A $k$-factor-critical graph $G$ is said to be \emph{minimal} if $G-e$ is not
$k$-factor-critical for any edge $e$ in $G$ \cite{favaron1996minimally}. It is easy to see that if $G$ is a $k$-factor-critical graph, then
$\delta(G)\ge k+1$. Favaron and Shi \cite{favaron1996minimally} conjectured that if $G$
is moreover a minimal $k$-factor-critical graph, then it satisfies $\delta(G)=k+1.$

\begin{conjecture}
	[\cite{favaron1996minimally,zhang2010equivalence}\label{conjetura}] Let $G$ be a minimal $k$-factor-critical
	graph. Then $\delta(G)=k+1.$
\end{conjecture}

Favaron and Shi \cite{favaron1996minimally} confirmed this conjecture for $k\in\{n-6,n-4,n-2\}$.
Guo et al. \cite{GUO202444,Guo2022MinimallyKG,GUO2024113839} confirmed this conjecture for $k\in\{2,n-8,n-10\}$,
and for claw-free graphs \cite{guo2023minimum}. Qiuli--Fuliang--Heping confirmed this
conjecture for planar graphs \cite{li2025minimal}.

Since $k<n$, if $\delta(G)\le k$ we can delete $k$ vertices, among which are all the neighbors of a vertex of minimum degree, and obtain
a graph with isolated vertices, which has no $\{1,2\}$-factor.
This proves the following observation.
\begin{observation}\label{obserpokasdpokm1}
	If $G$ is a $k$-$\{1,2\}$-factor-critical
	graph, then $\delta(G)\ge k+1$. 
\end{observation}

A $k$-$\{1,2\}$-factor-critical graph is said to be \emph{minimal} if $G-e$ is not
$k$-$\{1,2\}$-factor-critical for any edge $e$ in
$G$.

In contrast to \cref{conjetura} for minimal $k$-factor-critical
graphs, minimal $k$-$\{1,2\}$-factor-critical graphs may have
minimum degree $k+2$, regardless of the order of the graph. For instance,
for $n\ge4$ the complete graph $K_{n}$ is a minimal $(n-3)$-$\{1,2\}$-factor-critical
graph, while $\delta(K_{n})=k+2=n-1$. Moreover, $K_{n}$
is a minimal $(n-2)$-$\{1,2\}$-factor-critical graph.

\begin{conjecture}[\cite{KEVINksachsminimal1}\label{pokk123kjk}]
	Let $G$ be a minimal $k$-$\{1,2\}$-factor-critical
	graph. Then $k+1\le\delta(G)\le k+2.$
\end{conjecture}

In \cite{KEVINksachsminimal1} it is shown that \cref{pokk123kjk} holds for some infinite families.

\begin{theorem}[\cite{KEVINksachsminimal1}\label{asd12300}]
	Let $k\in\{0,1,n-5,n-4,n-3,n-2\}$ and let $G$ be a minimal
	$k$-$\{1,2\}$-factor-critical graph. Then $k+1\le\delta(G)\le k+2.$
\end{theorem}

\begin{theorem}[\cite{KEVINksachsminimal1}]
	Let $G$ be a minimal $k$-$\{1,2\}$-factor-critical
	claw-free graph of order $n$ such that $n-k$ is even and $\kappa(G)>k$. Then $\delta(G)=k+1$. 
\end{theorem}

In this paper, we show that \cref{pokk123kjk} holds for $k$-planar graphs, that is, graphs in which the deletion of any set of $k$ vertices yields a planar graph.

The paper is organized as follows. In \cref{sss0} we present the general context of the problem and introduce the fundamental concepts. In \cref{sss1} we fix the notation that will be used throughout the article. Finally, in \cref{sss2} we state and prove the main results.

\section{Preliminaries}\label{sss1}
All graphs considered in this paper are finite, undirected, and simple. 
For any undefined terminology or notation, we refer the reader to 
Lovász and Plummer \cite{LP} or Diestel \cite{Distel}.

Let \( G = (V, E) \) be a simple graph, where \( V = V(G) \) is the finite set of vertices and \( E = E(G) \) is the set of edges, with \( E \subseteq \{\{u, v\} : u, v \in V, u \neq v\} \). We denote the edge \( e=\{u, v\} \) as \( uv \). A subgraph of \( G \) is a graph \( H \) such that \( V(H) \subseteq V(G) \) and \( E(H) \subseteq E(G) \). A subgraph \( H \) of \( G \) is called a \textit{spanning} subgraph if \( V(H) = V(G) \). 

Let \( e \in E(G) \) and \( v \in V(G) \). We define \( G - e := (V, E \setminus \{e\}) \) and \( G - v := (V \setminus \{v\}, \{uw \in E : u,w \neq v\}) \). If \( X \subseteq V(G) \), the \textit{induced} subgraph of \( G \) by \( X \) is the subgraph \( G[X]=(X,F) \), where \( F:=\{uv \!\in\! E(G) : u, v \!\in \! X\} \).

The number of vertices in a graph $G$ is called the \textit{order} of the graph and denoted by $\left|G\right|$.
A \textit{cycle} in $G$ is called \textit{odd} (resp. \textit{even}) if it has an odd (resp. even) number of edges.

For a vertex $v\in V(G)$, the \emph{neighborhood} of $v$ is
\[
N_G(v)=\{u\in V(G): uv\in E(G)\}.
\]
When no confusion arises, we write $N(v)$ instead of $N_G(v)$. For a set $S\subseteq V(G)$, the \emph{neighborhood} of $S$ is
\[
N_G(S)=\bigcup_{v\in S} N_G(v).
\]

\noindent The \emph{degree} of a vertex $v\in V(G)$ is
$
\deg_G(v)=|N_G(v)|.
$
\noindent The \emph{minimum degree} of $G$ is
$
\delta(G)=\min\{\deg_G(v): v\in V(G)\}.
$
A graph $G$ is called \emph{$r$-regular} if $\deg_G(v)=r$ for every
$v\in V(G)$. A vertex $v\in V(G)$ is called an \emph{isolated vertex} if
$
\deg_G(v)=0.
$
The number of isolated vertices of a graph $G$ is denoted by $i(G)$.

A \textit{matching} \(M\) in a graph \(G\) is a set of pairwise non-adjacent edges. The \textit{matching number} of \(G\), denoted by  \(\mu(G)\), is the maximum cardinality of any matching in \(G\). Matchings induce an involution on the vertex set of the graph: \(M:V(G)\rightarrow V(G)\), where \(M(v)=u\) if \(uv \in M\), and \(M(v)=v\) otherwise. If \(S, U \subseteq V(G)\) with \(S \cap U = \emptyset\), we say that \(M\) is a matching from \(S\) to \(U\) if \(M(S) \subseteq U\). A matching $M$ is \emph{perfect} if $M(v)\neq v$ for every vertex
of the graph.

A vertex set \( S \subseteq V \) is \textit{independent} if, for every pair of vertices \( u, v \in S \), we have \( uv \notin E \). 
The number of vertices in a maximum independent set is denoted by \( \alpha(G) \).  A \textit{bipartite} graph is a graph whose vertex set can be partitioned into two disjoint independent sets.

\section{Main Results}\label{sss2}

We briefly recall the notion of graph minors. A graph $H$ is said
to be a minor of a graph $G$ if $H$ can be obtained from $G$ by
a sequence of vertex deletions, edge deletions, and edge contractions.
A graph is called \emph{planar} if it admits a drawing in the plane
in which no two edges intersect except possibly at their common endpoints.
The complete graph on $n$ vertices and the complete bipartite graph
with parts of size $n$ and $m$ are denoted and $K_{n}$ , $K_{3,3}$,
respectively. The following classic theorem, named Kuratowski\textquoteright s
theorem, gives a characterization of planar graphs.

\begin{theorem}
	 [\cite{kuratowski1930probleme,wagner1937eigenschaft}\label{asdij120io3}] A graph $G$ is planar if and
only if it contains neither $K_{5}$ nor $K_{3,3}$ as a minor. 
\end{theorem}

A graph is said to be $k$-planar if $G-S$ is planar for every $S\s V(G)$ with
$\a S=k$. In particular, every planar graph of order $n$ is $k$-planar
for each $k=0,\dots,n$. In \cref{1231ss12} a non-planar graph is shown,
which can be verified by applying \cref{asdij120io3}. However, as observed
in \cref{1231szzzs12}, the deletion of any vertex yields
a planar subgraph. Consequently, this non-planar graph turns out to be
$1$-planar.

\begin{figure}[H]
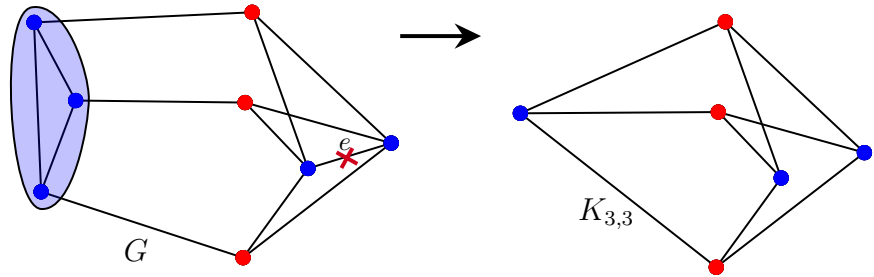

	
	\begin{center}

		\tikzset{every picture/.style={line width=0.75pt}} 
		


	\end{center}

	\caption{Example of a non-planar graph}
	
	\label{1231ss12}
	
\end{figure}

\begin{figure}[H]
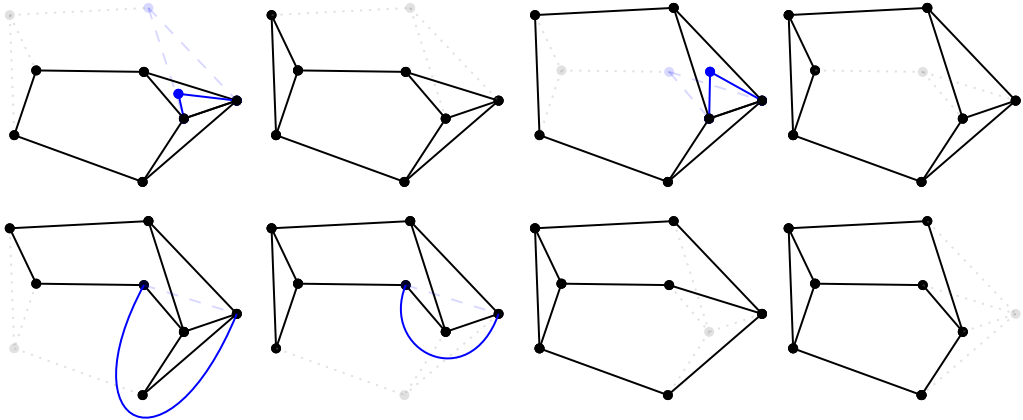

	
	\begin{center}

		\tikzset{every picture/.style={line width=0.75pt}} 
		
		%

	\end{center}

	\caption{Example of a non-planar and $1$-planar graph}
	
	\label{1231szzzs12}
	
\end{figure}

\begin{lemma}[\label{apoksjasd}\cite{KEVINcuatriupla}] $G$ is a $1$-$\{1,2\}$-factor-critical graph if
	and only if $\a S<\a{N(S)}$ for every nonempty independent set
	$S\s V(G)$.
\end{lemma}

Graphs satisfying the condition of \cref{apoksjasd} are known as $2$-bicritical graphs, originally introduced in \cite{pulleyblank1979minimum}. The class of $2$-bicritical graphs can be regarded as the structural counterpart of König--Egerváry graphs \cite{larson2011critical}. In recent works, several new properties of $2$-bicritical graphs have been established; see, for instance, \cite{kevinSDKECHAR, kevinSDKEGE, kevinPOSYFACTOR}.

\begin{lemma}
	[\label{lm123lm32}\cite{KEVINcuatriupla}] $G$ has a $\{1,2\}$-factor if and only if $\a S\le\a{N(S)}$
	for every set $S\s V(G)$.
\end{lemma}

The number $d_{G}(X)=\a X-\a{N(X)}$ is the \emph{difference} of the set
$X\s V(G)$, and $d(G)=\max\left\{ d_{G}(X):X\s V(G)\right\}$ is called
the \emph{critical difference} of $G$. A set $U\s V(G)$ is
\emph{critical} if $d_{G}(U)=d(G)$ \cite{zhang1990finding}. The number
$id(G)=\max\left\{ d_{G}(I): I\text{ is an independent set of }G\right\}$
is called the \emph{critical independence difference} of $G$. If
$I\s V(G)$ is independent and $d_{G}(I)=id(G)$, then $I$ is called
\emph{critical independent} \cite{zhang1990finding}. It is known that the equality
$d(G)=id(G)$ holds for every graph $G$ \cite{zhang1990finding}.

\begin{lemma}
	[\cite{zhang1990finding}\label{ijn12311}] For every graph $G$, we have $d(G)=id(G)$.
\end{lemma}

A bipartite graph with bipartition $A,B$ is said to be balanced if $\a A=\a B.$

\begin{lemma}\label{asd123sss}
	Let $G$ be a planar graph with bipartition $A,B$ and
	balanced. Then there exists a vertex $v\in A$ such that $\deg(v)\le3$. 
\end{lemma}

\begin{proof}
	Let $f$ be the number of faces of $G$, $n=\a G$, and $m=\left\Vert G\right\Vert$.
	In a planar bipartite graph, each face has length at least $4$,
	therefore $2m\ge4f$. Moreover, it is well known that planar graphs satisfy
	Euler’s formula: $n-m+f=2$. Hence $2m\ge4(2-n+m)=8-4n+4m$.
	From this it follows that $m\le2n-4$. By contradiction, suppose that $\deg(v)\ge4$
	for every vertex $v\in A$. Then the sum of degrees satisfies
	\[
	\sm_{v\in A}\deg(v)=m\ge4\a A=2n.
	\]
	\noindent This contradicts the fact that $m\le2n-4$.
\end{proof}

\begin{lemma}\label{asd123sss2}
	Let $G$ be a graph such that $G-e$ is a planar bipartite graph
	with bipartition $A,B$ satisfying $\a B+1\Le\a A\le\a B+2$, where $e\s A$.
	Then there exists a vertex $v\in A$ such that $\deg_{G}(v)\le3$. 
\end{lemma}

\begin{proof}
	Let $n=\a G=\a{G-e}$, $m(G)=\left\Vert G\right\Vert$, and $m(G-e)=m(G)-1$.
	Then, as in the proof of \cref{asd123sss}, we have $m(G-e)\le2n-4$.
	By contradiction, suppose that $\deg_{G}(v)\ge4$ for every vertex
	$v\in A$. Then the sum of degrees satisfies
	\begin{eqnarray*}
		2n+2 & = & 2\a A+2\left(\a B+1\right)\\
		& \Le & 4\a A\\
		& \Le & \sm_{v\in A}\deg_{G}(v)\\
		& = & \sm_{v\in A}\deg_{G-e}(v)+2\\
		& \le & m(G-e)+2.
	\end{eqnarray*}
	\noindent That is, $2n\Le m(G-e)$. This contradicts the fact that
	$m(G-e)\le2n-4$.
\end{proof}

\begin{lemma}\label{asd123sss3}
	Let $G$ be a planar graph with bipartition $A,B$ such that
	$\a A=\a B-1$. Then there exists a vertex $v\in A$ such that $\deg(v)\le3$.
\end{lemma} 

\begin{proof}
	Let $n=\a G$ and $m=\left\Vert G\right\Vert$. Then, as in the
	proof of \cref{asd123sss}, we have $m\le2n-4$. By contradiction, suppose
	that $\deg(v)\ge4$ for every vertex $v\in A$. Then the sum of
	degrees satisfies
	\[
	\sm_{v\in A}\deg(v)=m\ge4\a A=2n-2.
	\]
	\noindent This contradicts the fact that $m\le2n-4$.
\end{proof}

\begin{lemma}\label{asd123sss4}
	Let $G$ be a graph such that $G-e$ is a planar bipartite graph
	with bipartition $A,B$ satisfying $\a B\Le\a A\le\a B+1$, where $e\s A$.
	Then there exists a vertex $v\in A$ such that $\deg_{G}(v)\le3$. 
\end{lemma}

\begin{proof}
	Let $n=\a G=\a{G-e}$, $m(G)=\left\Vert G\right\Vert$, and $m(G-e)=m(G)-1$.
	Then, as in the proof of \cref{asd123sss}, we have $m(G-e)\le2n-4$.
	By contradiction, suppose that $\deg_{G}(v)\ge4$ for every vertex
	$v\in A$. Then the sum of degrees satisfies
	\begin{eqnarray*}
		2n & \Le & 2\a A+2\a B\\
		& \Le & 4\a A\\
		& \Le & \sm_{v\in A}\deg_{G}(v)\\
		& = & \sm_{v\in A}\deg_{G-e}(v)+2\\
		& \le & m(G-e)+2.
	\end{eqnarray*}
	\noindent That is, $2n-2\Le m(G-e)$. This contradicts the fact that
	$m(G-e)\le2n-4$.
\end{proof}

\begin{theorem}\label{main}
	Let $G$ be a $k$-$\{1,2\}$-factor-critical graph and suppose
	that there exists an edge $e\in E(G)$ such that $G-e$ is not a
	$k$-$\{1,2\}$-factor-critical graph. Then
	\begin{itemize}
		\item If $k=0$ and $G$ is planar, then $k+1\le\delta(G)\le k+3$.
		\item If $k>0$ and $G$ is $k$-planar, then $k+1\le\delta(G)\le k+2$.
	\end{itemize}
\end{theorem}

\begin{proof}
	The lower bounds in both items follow directly from
	\cref{obserpokasdpokm1}. Therefore, in what follows we focus only
	on proving the upper bounds.
	
	Let $k=0$ and suppose that $G$ has a $\{1,2\}$-factor but $G-e$
	does not. By \cref{lm123lm32} there exists a set $S^{\P}\s V(G)$ such that
	$\a{S^{\P}}>\a{N_{G-e}(S^{\P})}$, which implies that
	$d(G-e)\ge d_{G-e}(S^{\P})=\a{S^{\P}}-\a{N_{G-e}(S^{\P})}>0$.
	Then, by \cref{ijn12311}, there exists an independent set $S\s V(G)$
	of $G-e$ such that $\a S>\a{N_{G-e}(S)}$. By \cref{lm123lm32} we have
	$\a S\le\a{N_{G}(S)}$.
	Thus we conclude that
	\[
	\a{N_{G}(S)}-2\Le\a{N_{G-e}(S)}<\a S\Le\a{N_{G}(S)},
	\]
	\noindent that is, $\a{N_{G}(S)}-1\le\a S\Le\a{N_{G}(S)}$. Now
	let $e=xy$. Since $\a{N_{G-e}(S)}<\a{N_{G}(S)}$, it follows that
	at least one of the vertices $x$ or $y$ belongs to $S$.
	Without loss of generality, we consider the following two cases.
	
	$ $\\
	\textbf{Case 1.} Suppose that $x\in S$ and $y\notin S$. In this
	case, $S$ is also an independent set in $G$ and $y\notin N_{G-e}(S)$,
	since otherwise $\a{N_{G-e}(S)}=\a{N_{G}(S)}$, a contradiction.
	Hence,
	\[
	\a S\Le\a{N_{G}(S)}=\a{N_{G-e}(S)}+1\Le\a S,
	\]
	\noindent that is, $N_{G}(S)=N_{G-e}(S)\u\{y\}$ and
	$\a{N_{G}(S)}=\a S$. Therefore, the graph $H$ obtained from
	$G[S\u N_{G}(S)]$ by removing edges joining vertices of $N_{G}(S)$
	is a balanced planar bipartite graph with bipartition
	$S,N_{G}(S)$. By \cref{asd123sss}, there exists a vertex
	$v\in S$ such that $\deg_{H}(v)\le3$. Note that
	\[
	\delta(G)\Le\deg_{G}(v)=\deg_{H}(v)\Le3=k+3.
	\]
	\noindent As desired.
	
	$ $\\
	\textbf{Case 2.} Suppose that $x\in S$ and $y\in S$. In this case,
	\[
	N_{G-e}(S)\u\{x,y\}=N_{G}(S)\text{ and }\a{N_{G}(S)}=\a{N_{G-e}(S)}+2.
	\]
	\noindent Let $H$ be the graph obtained from $G[S\u N_{G}(S)]$ by
	removing edges joining vertices of $N_{G}(S)-\{x,y\}$. Then $H-e$
	is a planar bipartite graph with bipartition $S,N_{H-e}(S)$, where
	\[
	\a{N_{H-e}(S)}+1\Le\a S\Le\a{N_{H-e}(S)}+2.
	\]
	\noindent Therefore, by \cref{asd123sss2}, there exists a vertex
	$v\in S$ such that $\deg_{H}(v)\Le3$. As before, this implies that
	$\delta(G)\Le\deg_{G}(v)=\deg_{H}(v)\Le3=k+3.$
	
	$ $
	
	Now suppose that $k>0$. Since $G-e$ is not a $k$-$\{1,2\}$-factor-critical
	graph, there exists a set $S^{\P}\s V(G)$ with $\a{S^{\P}}=k$
	such that $\left(G-e\right)-S^{\P}$ has no $\{1,2\}$-factor.
	On the other hand, let $v\in S^{\P}$, let $S=S^{\P}-\{v\}$,
	and define $G^{\P}=G-S$.
	
	\begin{claim}
		$G^{\P}$ is a planar graph.
	\end{claim}
	
	\begin{claim}\label{claim1}
		$G^{\P}$ is a $1$-$\{1,2\}$-factor-critical
		graph.
	\end{claim}
	
	\begin{proof}
		Note that for every $x\in V(G-S)$, the graph
		$\left(G-S\right)-x=G-\left(S\u\{x\}\right)$
		has a $\{1,2\}$-factor, since $\a{S\u\{x\}}=k$ and $G$ is a
		$k$-$\{1,2\}$-factor-critical graph.
	\end{proof}
	
	\begin{claim}\label{claim2}
		$G^{\P}-e$ is not a $1$-$\{1,2\}$-factor-critical
		graph.
	\end{claim}
	
	\begin{proof}
		Note that $\left(G^{\P}-e\right)-v=\left(\left(G-S\right)-e\right)-v
		=\left(G-e\right)-S^{\P}$ has no $\{1,2\}$-factor.
	\end{proof}

Then, by \cref{claim2} and \cref{apoksjasd}, there exists a nonempty independent
set $S\s V(G)$ such that $\a S\ge\a{N_{G^{\P}-e}(S)}$. On the other hand, by
\cref{claim1} and \cref{apoksjasd}, we have $\a S<\a{N_{G^{\P}}(S)}$ (note that this holds
regardless of whether $S$ is independent in $G^{\P}$). Now, denoting
$e=xy$, it follows that at least one of the vertices $x$ or $y$ belongs
to $S$. Without loss of generality, we consider the following two cases.

$ $\\
\textbf{Case 1.} Suppose that $x\in S$ and $y\notin S$. In this
case, $S$ is also an independent set in $G$ and
$\a{N_{G^{\P}-e}(S)}=\a S=\a{N_{G^{\P}}(S)}-1$.
Therefore, the graph $H$ obtained from $G[S\u N_{G^{\P}}(S)]$ by
removing edges joining vertices of $N_{G^{\P}}(S)$ is bipartite
and planar with bipartition $S,N_{G^{\P}}(S)$ and
$\a S=\a{N_{G^{\P}}(S)}-1$.
By \cref{asd123sss3}, there exists a vertex $v\in S$ such that
$\deg_{H}(v)\le3$. Since $\deg_{H}(v)=\deg_{G^{\P}}(v)$, we obtain
\[
\delta(G)\Le\deg_{G^{\P}}(v)+\a S=\deg_{H}(v)+k-1\Le k+2.
\]
\noindent As desired.

$ $\\
\textbf{Case 2.} Suppose that $x\in S$ and $y\in S$. In this case,
\[
N_{G}(S)\u\{x,y\}=N_{G-e}(S)\text{ and }\a{N_{G}(S)}=\a{N_{G-e}(S)}+2.
\]

\noindent Let $H$ be the graph obtained from $G[S\u N_{G}(S)]$ by
removing edges joining vertices of $N_{G}(S)-\{x,y\}$. Then $H-e$
is a bipartite graph with bipartition $S,N_{H-e}(S)$ satisfying
\[
\a{N_{H-e}(S)}\Le\a S\le\a{N_{H-e}(S)}+1.
\]
\noindent Therefore, by \cref{asd123sss4}, there exists a vertex
$v\in S$ such that $\deg_{H}(v)\Le3$. As before, this implies that
$\delta(G)\Le\deg_{H}(v)\le k+3$.
\end{proof}

The bounds for $\delta(G)$ in \cref{main} are attainable. For
$k>0$, see \cref{1231sa113431}. The graph on the left in
\cref{1231sa113431} is a $1$-$\{1,2\}$-factor-critical graph whereas
$G-e$ is not; moreover, $\delta(G)=k+2=3$. The graph on the right is a
$1$-$\{1,2\}$-factor-critical graph whereas $G-e$ is not; moreover,
$\delta(G)=k+1=2$. Both graphs are planar, and therefore also
$1$-planar. For $k=0$, consider $K_{2}$, $K_{3}$, or the graph in
\cref{asd1w312312}; all three graphs are planar
$0$-$\{1,2\}$-factor-critical graphs with an edge whose deletion breaks
the existence of $\{1,2\}$-factors and
$\delta(G)=k+1,k+2,k+3$, respectively.

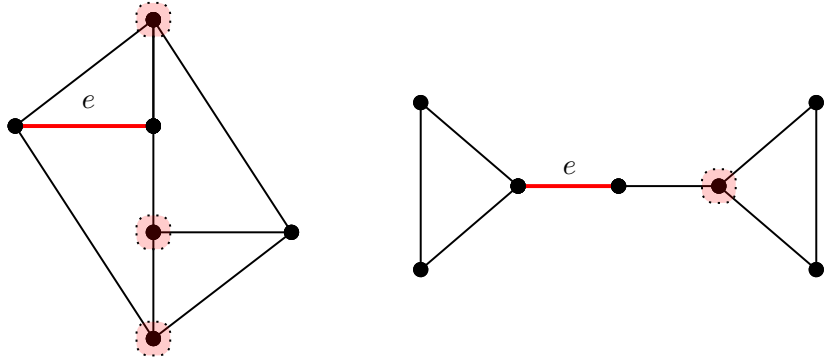
\begin{figure}[H]
	
	\begin{center}

\tikzset{every picture/.style={line width=0.75pt}} 

\begin{tikzpicture}[x=0.75pt,y=0.75pt,yscale=-1,xscale=1]
	
	\draw    (153.67,121.57) -- (153.67,175.21) ;
	\draw [shift={(153.67,175.21)}, rotate = 90] [color={rgb, 255:red, 0; green, 0; blue, 0 }  ][fill={rgb, 255:red, 0; green, 0; blue, 0 }  ][line width=0.75]      (0, 0) circle [x radius= 3.35, y radius= 3.35]   ;
	\draw [shift={(153.67,121.57)}, rotate = 90] [color={rgb, 255:red, 0; green, 0; blue, 0 }  ][fill={rgb, 255:red, 0; green, 0; blue, 0 }  ][line width=0.75]      (0, 0) circle [x radius= 3.35, y radius= 3.35]   ;
	\draw    (153.67,175.21) -- (153.67,228.86) ;
	\draw [shift={(153.67,228.86)}, rotate = 90] [color={rgb, 255:red, 0; green, 0; blue, 0 }  ][fill={rgb, 255:red, 0; green, 0; blue, 0 }  ][line width=0.75]      (0, 0) circle [x radius= 3.35, y radius= 3.35]   ;
	\draw [shift={(153.67,175.21)}, rotate = 90] [color={rgb, 255:red, 0; green, 0; blue, 0 }  ][fill={rgb, 255:red, 0; green, 0; blue, 0 }  ][line width=0.75]      (0, 0) circle [x radius= 3.35, y radius= 3.35]   ;
	\draw    (153.67,67.92) -- (153.67,121.57) ;
	\draw [shift={(153.67,121.57)}, rotate = 90] [color={rgb, 255:red, 0; green, 0; blue, 0 }  ][fill={rgb, 255:red, 0; green, 0; blue, 0 }  ][line width=0.75]      (0, 0) circle [x radius= 3.35, y radius= 3.35]   ;
	\draw [shift={(153.67,67.92)}, rotate = 90] [color={rgb, 255:red, 0; green, 0; blue, 0 }  ][fill={rgb, 255:red, 0; green, 0; blue, 0 }  ][line width=0.75]      (0, 0) circle [x radius= 3.35, y radius= 3.35]   ;
	\draw    (153.67,121.57) -- (153.67,67.92) ;
	\draw [shift={(153.67,67.92)}, rotate = 270] [color={rgb, 255:red, 0; green, 0; blue, 0 }  ][fill={rgb, 255:red, 0; green, 0; blue, 0 }  ][line width=0.75]      (0, 0) circle [x radius= 3.35, y radius= 3.35]   ;
	\draw [shift={(153.67,121.57)}, rotate = 270] [color={rgb, 255:red, 0; green, 0; blue, 0 }  ][fill={rgb, 255:red, 0; green, 0; blue, 0 }  ][line width=0.75]      (0, 0) circle [x radius= 3.35, y radius= 3.35]   ;
	\draw    (153.67,175.21) -- (223.33,175.21) ;
	\draw [shift={(223.33,175.21)}, rotate = 0] [color={rgb, 255:red, 0; green, 0; blue, 0 }  ][fill={rgb, 255:red, 0; green, 0; blue, 0 }  ][line width=0.75]      (0, 0) circle [x radius= 3.35, y radius= 3.35]   ;
	\draw [shift={(153.67,175.21)}, rotate = 0] [color={rgb, 255:red, 0; green, 0; blue, 0 }  ][fill={rgb, 255:red, 0; green, 0; blue, 0 }  ][line width=0.75]      (0, 0) circle [x radius= 3.35, y radius= 3.35]   ;
	\draw    (153.67,121.57) ;
	\draw [shift={(153.67,121.57)}, rotate = 0] [color={rgb, 255:red, 0; green, 0; blue, 0 }  ][fill={rgb, 255:red, 0; green, 0; blue, 0 }  ][line width=0.75]      (0, 0) circle [x radius= 3.35, y radius= 3.35]   ;
	\draw [shift={(153.67,121.57)}, rotate = 0] [color={rgb, 255:red, 0; green, 0; blue, 0 }  ][fill={rgb, 255:red, 0; green, 0; blue, 0 }  ][line width=0.75]      (0, 0) circle [x radius= 3.35, y radius= 3.35]   ;
	\draw    (153.67,228.86) -- (223.33,175.21) ;
	\draw [shift={(223.33,175.21)}, rotate = 322.4] [color={rgb, 255:red, 0; green, 0; blue, 0 }  ][fill={rgb, 255:red, 0; green, 0; blue, 0 }  ][line width=0.75]      (0, 0) circle [x radius= 3.35, y radius= 3.35]   ;
	\draw [shift={(153.67,228.86)}, rotate = 322.4] [color={rgb, 255:red, 0; green, 0; blue, 0 }  ][fill={rgb, 255:red, 0; green, 0; blue, 0 }  ][line width=0.75]      (0, 0) circle [x radius= 3.35, y radius= 3.35]   ;
	\draw    (153.67,67.92) -- (223.33,175.21) ;
	\draw [shift={(223.33,175.21)}, rotate = 57] [color={rgb, 255:red, 0; green, 0; blue, 0 }  ][fill={rgb, 255:red, 0; green, 0; blue, 0 }  ][line width=0.75]      (0, 0) circle [x radius= 3.35, y radius= 3.35]   ;
	\draw [shift={(153.67,67.92)}, rotate = 57] [color={rgb, 255:red, 0; green, 0; blue, 0 }  ][fill={rgb, 255:red, 0; green, 0; blue, 0 }  ][line width=0.75]      (0, 0) circle [x radius= 3.35, y radius= 3.35]   ;
	\draw    (84,121.57) -- (153.67,228.86) ;
	\draw [shift={(153.67,228.86)}, rotate = 57] [color={rgb, 255:red, 0; green, 0; blue, 0 }  ][fill={rgb, 255:red, 0; green, 0; blue, 0 }  ][line width=0.75]      (0, 0) circle [x radius= 3.35, y radius= 3.35]   ;
	\draw [shift={(84,121.57)}, rotate = 57] [color={rgb, 255:red, 0; green, 0; blue, 0 }  ][fill={rgb, 255:red, 0; green, 0; blue, 0 }  ][line width=0.75]      (0, 0) circle [x radius= 3.35, y radius= 3.35]   ;
	\draw    (84,121.57) -- (153.67,67.92) ;
	\draw [shift={(153.67,67.92)}, rotate = 322.4] [color={rgb, 255:red, 0; green, 0; blue, 0 }  ][fill={rgb, 255:red, 0; green, 0; blue, 0 }  ][line width=0.75]      (0, 0) circle [x radius= 3.35, y radius= 3.35]   ;
	\draw [shift={(84,121.57)}, rotate = 322.4] [color={rgb, 255:red, 0; green, 0; blue, 0 }  ][fill={rgb, 255:red, 0; green, 0; blue, 0 }  ][line width=0.75]      (0, 0) circle [x radius= 3.35, y radius= 3.35]   ;
	\draw [color={rgb, 255:red, 255; green, 0; blue, 0 }  ,draw opacity=1 ][line width=1.5]    (84,121.57) -- (153.67,121.57) ;
	\draw    (84,121.57) ;
	\draw [shift={(84,121.57)}, rotate = 0] [color={rgb, 255:red, 0; green, 0; blue, 0 }  ][fill={rgb, 255:red, 0; green, 0; blue, 0 }  ][line width=0.75]      (0, 0) circle [x radius= 3.35, y radius= 3.35]   ;
	\draw [shift={(84,121.57)}, rotate = 0] [color={rgb, 255:red, 0; green, 0; blue, 0 }  ][fill={rgb, 255:red, 0; green, 0; blue, 0 }  ][line width=0.75]      (0, 0) circle [x radius= 3.35, y radius= 3.35]   ;
	\draw    (153.67,67.92) ;
	\draw [shift={(153.67,67.92)}, rotate = 0] [color={rgb, 255:red, 0; green, 0; blue, 0 }  ][fill={rgb, 255:red, 0; green, 0; blue, 0 }  ][line width=0.75]      (0, 0) circle [x radius= 3.35, y radius= 3.35]   ;
	\draw [shift={(153.67,67.92)}, rotate = 0] [color={rgb, 255:red, 0; green, 0; blue, 0 }  ][fill={rgb, 255:red, 0; green, 0; blue, 0 }  ][line width=0.75]      (0, 0) circle [x radius= 3.35, y radius= 3.35]   ;
	\draw  [fill={rgb, 255:red, 255; green, 0; blue, 0 }  ,fill opacity=0.2 ][dash pattern={on 0.84pt off 2.51pt}] (145.17,172.71) .. controls (145.17,169.4) and (147.85,166.71) .. (151.17,166.71) -- (156.17,166.71) .. controls (159.48,166.71) and (162.17,169.4) .. (162.17,172.71) -- (162.17,177.71) .. controls (162.17,181.03) and (159.48,183.71) .. (156.17,183.71) -- (151.17,183.71) .. controls (147.85,183.71) and (145.17,181.03) .. (145.17,177.71) -- cycle ;
	\draw  [fill={rgb, 255:red, 255; green, 0; blue, 0 }  ,fill opacity=0.2 ][dash pattern={on 0.84pt off 2.51pt}] (145.17,65.42) .. controls (145.17,62.11) and (147.85,59.42) .. (151.17,59.42) -- (156.17,59.42) .. controls (159.48,59.42) and (162.17,62.11) .. (162.17,65.42) -- (162.17,70.42) .. controls (162.17,73.74) and (159.48,76.42) .. (156.17,76.42) -- (151.17,76.42) .. controls (147.85,76.42) and (145.17,73.74) .. (145.17,70.42) -- cycle ;
	\draw    (288.5,194.03) -- (288.5,109.73) ;
	\draw [shift={(288.5,109.73)}, rotate = 270] [color={rgb, 255:red, 0; green, 0; blue, 0 }  ][fill={rgb, 255:red, 0; green, 0; blue, 0 }  ][line width=0.75]      (0, 0) circle [x radius= 3.35, y radius= 3.35]   ;
	\draw [shift={(288.5,194.03)}, rotate = 270] [color={rgb, 255:red, 0; green, 0; blue, 0 }  ][fill={rgb, 255:red, 0; green, 0; blue, 0 }  ][line width=0.75]      (0, 0) circle [x radius= 3.35, y radius= 3.35]   ;
	\draw    (288.5,109.73) -- (337.63,151.88) ;
	\draw [shift={(337.63,151.88)}, rotate = 40.63] [color={rgb, 255:red, 0; green, 0; blue, 0 }  ][fill={rgb, 255:red, 0; green, 0; blue, 0 }  ][line width=0.75]      (0, 0) circle [x radius= 3.35, y radius= 3.35]   ;
	\draw [shift={(288.5,109.73)}, rotate = 40.63] [color={rgb, 255:red, 0; green, 0; blue, 0 }  ][fill={rgb, 255:red, 0; green, 0; blue, 0 }  ][line width=0.75]      (0, 0) circle [x radius= 3.35, y radius= 3.35]   ;
	\draw    (337.63,151.88) -- (288.5,194.03) ;
	\draw [shift={(288.5,194.03)}, rotate = 139.37] [color={rgb, 255:red, 0; green, 0; blue, 0 }  ][fill={rgb, 255:red, 0; green, 0; blue, 0 }  ][line width=0.75]      (0, 0) circle [x radius= 3.35, y radius= 3.35]   ;
	\draw [shift={(337.63,151.88)}, rotate = 139.37] [color={rgb, 255:red, 0; green, 0; blue, 0 }  ][fill={rgb, 255:red, 0; green, 0; blue, 0 }  ][line width=0.75]      (0, 0) circle [x radius= 3.35, y radius= 3.35]   ;
	\draw    (388.29,151.88) -- (337.63,151.88) ;
	\draw [shift={(337.63,151.88)}, rotate = 180] [color={rgb, 255:red, 0; green, 0; blue, 0 }  ][fill={rgb, 255:red, 0; green, 0; blue, 0 }  ][line width=0.75]      (0, 0) circle [x radius= 3.35, y radius= 3.35]   ;
	\draw [shift={(388.29,151.88)}, rotate = 180] [color={rgb, 255:red, 0; green, 0; blue, 0 }  ][fill={rgb, 255:red, 0; green, 0; blue, 0 }  ][line width=0.75]      (0, 0) circle [x radius= 3.35, y radius= 3.35]   ;
	\draw    (438.87,151.88) -- (388.29,151.88) ;
	\draw [shift={(388.29,151.88)}, rotate = 180] [color={rgb, 255:red, 0; green, 0; blue, 0 }  ][fill={rgb, 255:red, 0; green, 0; blue, 0 }  ][line width=0.75]      (0, 0) circle [x radius= 3.35, y radius= 3.35]   ;
	\draw [shift={(438.87,151.88)}, rotate = 180] [color={rgb, 255:red, 0; green, 0; blue, 0 }  ][fill={rgb, 255:red, 0; green, 0; blue, 0 }  ][line width=0.75]      (0, 0) circle [x radius= 3.35, y radius= 3.35]   ;
	\draw    (488,109.73) -- (488,194.03) ;
	\draw [shift={(488,194.03)}, rotate = 90] [color={rgb, 255:red, 0; green, 0; blue, 0 }  ][fill={rgb, 255:red, 0; green, 0; blue, 0 }  ][line width=0.75]      (0, 0) circle [x radius= 3.35, y radius= 3.35]   ;
	\draw [shift={(488,109.73)}, rotate = 90] [color={rgb, 255:red, 0; green, 0; blue, 0 }  ][fill={rgb, 255:red, 0; green, 0; blue, 0 }  ][line width=0.75]      (0, 0) circle [x radius= 3.35, y radius= 3.35]   ;
	\draw    (488,194.03) -- (438.87,151.88) ;
	\draw [shift={(438.87,151.88)}, rotate = 220.63] [color={rgb, 255:red, 0; green, 0; blue, 0 }  ][fill={rgb, 255:red, 0; green, 0; blue, 0 }  ][line width=0.75]      (0, 0) circle [x radius= 3.35, y radius= 3.35]   ;
	\draw [shift={(488,194.03)}, rotate = 220.63] [color={rgb, 255:red, 0; green, 0; blue, 0 }  ][fill={rgb, 255:red, 0; green, 0; blue, 0 }  ][line width=0.75]      (0, 0) circle [x radius= 3.35, y radius= 3.35]   ;
	\draw    (438.87,151.88) -- (488,109.73) ;
	\draw [shift={(488,109.73)}, rotate = 319.37] [color={rgb, 255:red, 0; green, 0; blue, 0 }  ][fill={rgb, 255:red, 0; green, 0; blue, 0 }  ][line width=0.75]      (0, 0) circle [x radius= 3.35, y radius= 3.35]   ;
	\draw [shift={(438.87,151.88)}, rotate = 319.37] [color={rgb, 255:red, 0; green, 0; blue, 0 }  ][fill={rgb, 255:red, 0; green, 0; blue, 0 }  ][line width=0.75]      (0, 0) circle [x radius= 3.35, y radius= 3.35]   ;
	\draw [color={rgb, 255:red, 255; green, 0; blue, 0 }  ,draw opacity=1 ][line width=1.5]    (337.63,151.88) -- (390.3,151.88) ;
	\draw    (337.63,151.88) ;
	\draw [shift={(337.63,151.88)}, rotate = 0] [color={rgb, 255:red, 0; green, 0; blue, 0 }  ][fill={rgb, 255:red, 0; green, 0; blue, 0 }  ][line width=0.75]      (0, 0) circle [x radius= 3.35, y radius= 3.35]   ;
	\draw [shift={(337.63,151.88)}, rotate = 0] [color={rgb, 255:red, 0; green, 0; blue, 0 }  ][fill={rgb, 255:red, 0; green, 0; blue, 0 }  ][line width=0.75]      (0, 0) circle [x radius= 3.35, y radius= 3.35]   ;
	\draw    (388.29,151.88) ;
	\draw [shift={(388.29,151.88)}, rotate = 0] [color={rgb, 255:red, 0; green, 0; blue, 0 }  ][fill={rgb, 255:red, 0; green, 0; blue, 0 }  ][line width=0.75]      (0, 0) circle [x radius= 3.35, y radius= 3.35]   ;
	\draw [shift={(388.29,151.88)}, rotate = 0] [color={rgb, 255:red, 0; green, 0; blue, 0 }  ][fill={rgb, 255:red, 0; green, 0; blue, 0 }  ][line width=0.75]      (0, 0) circle [x radius= 3.35, y radius= 3.35]   ;
	\draw  [fill={rgb, 255:red, 255; green, 0; blue, 0 }  ,fill opacity=0.2 ][dash pattern={on 0.84pt off 2.51pt}] (430.37,149.38) .. controls (430.37,146.07) and (433.06,143.38) .. (436.37,143.38) -- (441.37,143.38) .. controls (444.68,143.38) and (447.37,146.07) .. (447.37,149.38) -- (447.37,154.38) .. controls (447.37,157.69) and (444.68,160.38) .. (441.37,160.38) -- (436.37,160.38) .. controls (433.06,160.38) and (430.37,157.69) .. (430.37,154.38) -- cycle ;
	\draw    (153.67,121.57) ;
	\draw [shift={(153.67,121.57)}, rotate = 0] [color={rgb, 255:red, 0; green, 0; blue, 0 }  ][fill={rgb, 255:red, 0; green, 0; blue, 0 }  ][line width=0.75]      (0, 0) circle [x radius= 3.35, y radius= 3.35]   ;
	\draw [shift={(153.67,121.57)}, rotate = 0] [color={rgb, 255:red, 0; green, 0; blue, 0 }  ][fill={rgb, 255:red, 0; green, 0; blue, 0 }  ][line width=0.75]      (0, 0) circle [x radius= 3.35, y radius= 3.35]   ;
	\draw  [fill={rgb, 255:red, 255; green, 0; blue, 0 }  ,fill opacity=0.2 ][dash pattern={on 0.84pt off 2.51pt}] (145.17,226.36) .. controls (145.17,223.05) and (147.85,220.36) .. (151.17,220.36) -- (156.17,220.36) .. controls (159.48,220.36) and (162.17,223.05) .. (162.17,226.36) -- (162.17,231.36) .. controls (162.17,234.67) and (159.48,237.36) .. (156.17,237.36) -- (151.17,237.36) .. controls (147.85,237.36) and (145.17,234.67) .. (145.17,231.36) -- cycle ;
	
	\draw (116.02,104.8) node [anchor=north west][inner sep=0.75pt]  [font=\small]  {$e$};
	\draw (358.75,137.87) node [anchor=north west][inner sep=0.75pt]  [font=\small]  {$e$};

\end{tikzpicture}

	\end{center}

	\caption{Planar $k$-$\{1,2\}$-factor-critical graphs with $k=1$. In each example,
		$G-e$ is not a $k$-$\{1,2\}$-factor-critical graph and $\delta(G)=k+2=3$
		for the graph on the left, while $\delta(G)=k+1=2$ for the graph on the
		right. The highlighted vertices are sets $S\s V(G)$ satisfying
		$i(G-e-S)>\a S-k$, which verifies that $G-e$ is not a
		$k$-$\{1,2\}$-factor-critical graph by \cref{asd12392}.}
	
	\label{1231sa113431}
	
\end{figure}

\begin{figure}[H]
	
	\begin{center}

		\tikzset{every picture/.style={line width=0.75pt}} 
		
		\begin{tikzpicture}[x=0.75pt,y=0.75pt,yscale=-1,xscale=1]
			
			\draw    (171,136.84) -- (208.86,64.35) ;
			\draw [shift={(208.86,64.35)}, rotate = 297.57] [color={rgb, 255:red, 0; green, 0; blue, 0 }  ][fill={rgb, 255:red, 0; green, 0; blue, 0 }  ][line width=0.75]      (0, 0) circle [x radius= 3.35, y radius= 3.35]   ;
			\draw [shift={(171,136.84)}, rotate = 297.57] [color={rgb, 255:red, 0; green, 0; blue, 0 }  ][fill={rgb, 255:red, 0; green, 0; blue, 0 }  ][line width=0.75]      (0, 0) circle [x radius= 3.35, y radius= 3.35]   ;
			\draw    (208.86,64.35) -- (208.86,136.84) ;
			\draw [shift={(208.86,136.84)}, rotate = 90] [color={rgb, 255:red, 0; green, 0; blue, 0 }  ][fill={rgb, 255:red, 0; green, 0; blue, 0 }  ][line width=0.75]      (0, 0) circle [x radius= 3.35, y radius= 3.35]   ;
			\draw [shift={(208.86,64.35)}, rotate = 90] [color={rgb, 255:red, 0; green, 0; blue, 0 }  ][fill={rgb, 255:red, 0; green, 0; blue, 0 }  ][line width=0.75]      (0, 0) circle [x radius= 3.35, y radius= 3.35]   ;
			\draw    (208.86,136.84) -- (208.86,209.33) ;
			\draw [shift={(208.86,209.33)}, rotate = 90] [color={rgb, 255:red, 0; green, 0; blue, 0 }  ][fill={rgb, 255:red, 0; green, 0; blue, 0 }  ][line width=0.75]      (0, 0) circle [x radius= 3.35, y radius= 3.35]   ;
			\draw [shift={(208.86,136.84)}, rotate = 90] [color={rgb, 255:red, 0; green, 0; blue, 0 }  ][fill={rgb, 255:red, 0; green, 0; blue, 0 }  ][line width=0.75]      (0, 0) circle [x radius= 3.35, y radius= 3.35]   ;
			\draw    (252.35,140.86) -- (208.86,209.33) ;
			\draw [shift={(208.86,209.33)}, rotate = 122.43] [color={rgb, 255:red, 0; green, 0; blue, 0 }  ][fill={rgb, 255:red, 0; green, 0; blue, 0 }  ][line width=0.75]      (0, 0) circle [x radius= 3.35, y radius= 3.35]   ;
			\draw [shift={(252.35,140.86)}, rotate = 122.43] [color={rgb, 255:red, 0; green, 0; blue, 0 }  ][fill={rgb, 255:red, 0; green, 0; blue, 0 }  ][line width=0.75]      (0, 0) circle [x radius= 3.35, y radius= 3.35]   ;
			\draw    (252.35,140.86) -- (208.86,64.35) ;
			\draw [shift={(208.86,64.35)}, rotate = 240.39] [color={rgb, 255:red, 0; green, 0; blue, 0 }  ][fill={rgb, 255:red, 0; green, 0; blue, 0 }  ][line width=0.75]      (0, 0) circle [x radius= 3.35, y radius= 3.35]   ;
			\draw [shift={(252.35,140.86)}, rotate = 240.39] [color={rgb, 255:red, 0; green, 0; blue, 0 }  ][fill={rgb, 255:red, 0; green, 0; blue, 0 }  ][line width=0.75]      (0, 0) circle [x radius= 3.35, y radius= 3.35]   ;
			\draw    (295.04,156.97) -- (252.35,140.86) ;
			\draw [shift={(252.35,140.86)}, rotate = 200.67] [color={rgb, 255:red, 0; green, 0; blue, 0 }  ][fill={rgb, 255:red, 0; green, 0; blue, 0 }  ][line width=0.75]      (0, 0) circle [x radius= 3.35, y radius= 3.35]   ;
			\draw [shift={(295.04,156.97)}, rotate = 200.67] [color={rgb, 255:red, 0; green, 0; blue, 0 }  ][fill={rgb, 255:red, 0; green, 0; blue, 0 }  ][line width=0.75]      (0, 0) circle [x radius= 3.35, y radius= 3.35]   ;
			\draw    (208.86,209.33) -- (295.04,156.97) ;
			\draw [shift={(295.04,156.97)}, rotate = 328.72] [color={rgb, 255:red, 0; green, 0; blue, 0 }  ][fill={rgb, 255:red, 0; green, 0; blue, 0 }  ][line width=0.75]      (0, 0) circle [x radius= 3.35, y radius= 3.35]   ;
			\draw [shift={(208.86,209.33)}, rotate = 328.72] [color={rgb, 255:red, 0; green, 0; blue, 0 }  ][fill={rgb, 255:red, 0; green, 0; blue, 0 }  ][line width=0.75]      (0, 0) circle [x radius= 3.35, y radius= 3.35]   ;
			\draw    (208.86,209.33) -- (378,172.28) ;
			\draw [shift={(378,172.28)}, rotate = 347.64] [color={rgb, 255:red, 0; green, 0; blue, 0 }  ][fill={rgb, 255:red, 0; green, 0; blue, 0 }  ][line width=0.75]      (0, 0) circle [x radius= 3.35, y radius= 3.35]   ;
			\draw [shift={(208.86,209.33)}, rotate = 347.64] [color={rgb, 255:red, 0; green, 0; blue, 0 }  ][fill={rgb, 255:red, 0; green, 0; blue, 0 }  ][line width=0.75]      (0, 0) circle [x radius= 3.35, y radius= 3.35]   ;
			\draw    (208.86,64.35) -- (378,172.28) ;
			\draw [shift={(378,172.28)}, rotate = 32.54] [color={rgb, 255:red, 0; green, 0; blue, 0 }  ][fill={rgb, 255:red, 0; green, 0; blue, 0 }  ][line width=0.75]      (0, 0) circle [x radius= 3.35, y radius= 3.35]   ;
			\draw [shift={(208.86,64.35)}, rotate = 32.54] [color={rgb, 255:red, 0; green, 0; blue, 0 }  ][fill={rgb, 255:red, 0; green, 0; blue, 0 }  ][line width=0.75]      (0, 0) circle [x radius= 3.35, y radius= 3.35]   ;
			\draw    (295.04,156.97) -- (378,172.28) ;
			\draw [shift={(378,172.28)}, rotate = 10.45] [color={rgb, 255:red, 0; green, 0; blue, 0 }  ][fill={rgb, 255:red, 0; green, 0; blue, 0 }  ][line width=0.75]      (0, 0) circle [x radius= 3.35, y radius= 3.35]   ;
			\draw [shift={(295.04,156.97)}, rotate = 10.45] [color={rgb, 255:red, 0; green, 0; blue, 0 }  ][fill={rgb, 255:red, 0; green, 0; blue, 0 }  ][line width=0.75]      (0, 0) circle [x radius= 3.35, y radius= 3.35]   ;
			\draw    (171,136.84) -- (208.86,209.33) ;
			\draw [shift={(208.86,209.33)}, rotate = 62.43] [color={rgb, 255:red, 0; green, 0; blue, 0 }  ][fill={rgb, 255:red, 0; green, 0; blue, 0 }  ][line width=0.75]      (0, 0) circle [x radius= 3.35, y radius= 3.35]   ;
			\draw [shift={(171,136.84)}, rotate = 62.43] [color={rgb, 255:red, 0; green, 0; blue, 0 }  ][fill={rgb, 255:red, 0; green, 0; blue, 0 }  ][line width=0.75]      (0, 0) circle [x radius= 3.35, y radius= 3.35]   ;
			\draw    (208.86,136.84) -- (171,136.84) ;
			\draw [shift={(171,136.84)}, rotate = 180] [color={rgb, 255:red, 0; green, 0; blue, 0 }  ][fill={rgb, 255:red, 0; green, 0; blue, 0 }  ][line width=0.75]      (0, 0) circle [x radius= 3.35, y radius= 3.35]   ;
			\draw [shift={(208.86,136.84)}, rotate = 180] [color={rgb, 255:red, 0; green, 0; blue, 0 }  ][fill={rgb, 255:red, 0; green, 0; blue, 0 }  ][line width=0.75]      (0, 0) circle [x radius= 3.35, y radius= 3.35]   ;
			\draw [color={rgb, 255:red, 255; green, 0; blue, 0 }  ,draw opacity=1 ][line width=1.5]    (171,136.84) -- (208.86,136.84) ;
			\draw  [fill={rgb, 255:red, 255; green, 0; blue, 0 }  ,fill opacity=0.2 ][dash pattern={on 0.84pt off 2.51pt}] (202.01,63.5) .. controls (202.01,60.19) and (204.7,57.5) .. (208.01,57.5) -- (209.7,57.5) .. controls (213.02,57.5) and (215.7,60.19) .. (215.7,63.5) -- (215.7,65.19) .. controls (215.7,68.51) and (213.02,71.19) .. (209.7,71.19) -- (208.01,71.19) .. controls (204.7,71.19) and (202.01,68.51) .. (202.01,65.19) -- cycle ;
			\draw  [fill={rgb, 255:red, 255; green, 0; blue, 0 }  ,fill opacity=0.2 ][dash pattern={on 0.84pt off 2.51pt}] (202.01,208.48) .. controls (202.01,205.17) and (204.7,202.48) .. (208.01,202.48) -- (209.7,202.48) .. controls (213.02,202.48) and (215.7,205.17) .. (215.7,208.48) -- (215.7,210.17) .. controls (215.7,213.49) and (213.02,216.17) .. (209.7,216.17) -- (208.01,216.17) .. controls (204.7,216.17) and (202.01,213.49) .. (202.01,210.17) -- cycle ;
			\draw  [fill={rgb, 255:red, 255; green, 0; blue, 0 }  ,fill opacity=0.2 ][dash pattern={on 0.84pt off 2.51pt}] (288.19,156.13) .. controls (288.19,152.81) and (290.88,150.13) .. (294.19,150.13) -- (295.89,150.13) .. controls (299.2,150.13) and (301.89,152.81) .. (301.89,156.13) -- (301.89,157.82) .. controls (301.89,161.13) and (299.2,163.82) .. (295.89,163.82) -- (294.19,163.82) .. controls (290.88,163.82) and (288.19,161.13) .. (288.19,157.82) -- cycle ;
			\draw    (208.86,136.84) ;
			\draw [shift={(208.86,136.84)}, rotate = 0] [color={rgb, 255:red, 0; green, 0; blue, 0 }  ][fill={rgb, 255:red, 0; green, 0; blue, 0 }  ][line width=0.75]      (0, 0) circle [x radius= 3.35, y radius= 3.35]   ;
			\draw [shift={(208.86,136.84)}, rotate = 0] [color={rgb, 255:red, 0; green, 0; blue, 0 }  ][fill={rgb, 255:red, 0; green, 0; blue, 0 }  ][line width=0.75]      (0, 0) circle [x radius= 3.35, y radius= 3.35]   ;
			\draw    (171,136.84) ;
			\draw [shift={(171,136.84)}, rotate = 0] [color={rgb, 255:red, 0; green, 0; blue, 0 }  ][fill={rgb, 255:red, 0; green, 0; blue, 0 }  ][line width=0.75]      (0, 0) circle [x radius= 3.35, y radius= 3.35]   ;
			\draw [shift={(171,136.84)}, rotate = 0] [color={rgb, 255:red, 0; green, 0; blue, 0 }  ][fill={rgb, 255:red, 0; green, 0; blue, 0 }  ][line width=0.75]      (0, 0) circle [x radius= 3.35, y radius= 3.35]   ;
			
			\draw (187.74,126.36) node [anchor=north west][inner sep=0.75pt]  [font=\small]  {$e$};

		\end{tikzpicture}

	\end{center}

	\caption{Planar $0$-$\{1,2\}$-factor-critical graph with minimum degree
		$\delta(G)=0+3=3$. Moreover, $G-e$ is not a
		$0$-$\{1,2\}$-factor-critical graph, which can be verified using the
		highlighted set via \cref{asd12392}.}
	
	\label{asd1w312312}
	
\end{figure}
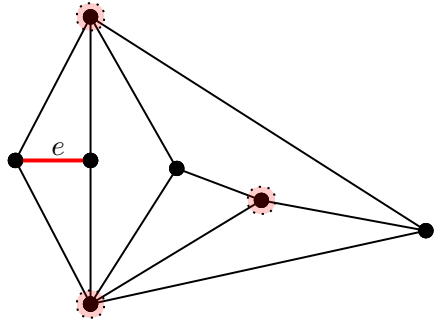

From \cref{main}, the following two results are obtained directly.

\begin{theorem}\label{main2}
	Let $G$ be a minimal $k$-$\{1,2\}$-factor-critical
	graph such that for every $S\s V(G)$ with $\a S=k-1$ the graph $G-S$
	is planar. Then $k+1\le\delta(G)\le k+2$. 
\end{theorem}

\begin{proof}
	By \cref{main}, it suffices to prove that every minimal
	$0$-$\{1,2\}$-factor-critical graph $G$ satisfies
	$1\le\delta(G)\le2$. But this is trivially true, since $G$
	itself is a $\{1,2\}$-factor in this case.
\end{proof}

The bounds in \cref{main2} are tight. Note that $K_{5}$ is a
minimal $2$-$\{1,2\}$-factor-critical graph for which $K_{5}-v$ is planar
for every $v\in V(K_{5})$, while $\delta(K_{5})=4=k+2$ with $k=2$.
On the other hand, the graph in \cref{1231ss12} is a $1$-planar graph,
and \cref{12312s13245fff} shows that it is a
$2$-$\{1,2\}$-factor-critical graph. Moreover, the minimality of this
graph as a $2$-$\{1,2\}$-factor-critical graph follows immediately:
removing an edge different from $e$ produces vertices of degree $2$,
yielding a graph that is not a $2$-$\{1,2\}$-factor-critical graph,
as ensured by \cref{obserpokasdpokm1}. On the other hand, $G-e$ is also
not a $2$-$\{1,2\}$-factor-critical graph, since if $S$ denotes the set
of red vertices in \cref{1231ss12}, then
$i(G-e-S)=2>\a S-2=1$, and hence $G-e$ is not
$2$-$\{1,2\}$-factor-critical by \cref{asd12392}. Finally, note that
$\delta(G)=k+1=3$.

\begin{figure}[H]
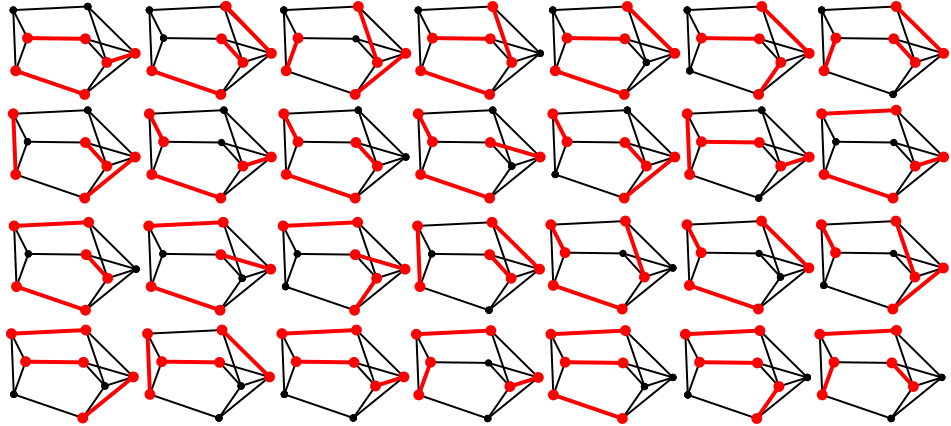

	
	\begin{center}

		\tikzset{every picture/.style={line width=0.75pt}} 
		


	\end{center}

	\caption{The graph in \cref{1231ss12}, \cref{1231szzzs12} is a
		$2$-$\{1,2\}$-factor-critical
		($2$-factor-critical) graph.}
	
	\label{12312s13245fff}
	
\end{figure}

\begin{theorem}\label{main3}
	Let $G$ be a minimal $k$-$\{1,2\}$-factor-critical
	planar graph. Then $k+1\le\delta(G)\le k+2$. 
\end{theorem}

If $k=1$ and $G=K_{4}$, we have $\delta(G)=k+2$, since
$G$ is a minimal $k$-$\{1,2\}$-factor-critical planar graph. If $k^{\P}=2$,
then $G$ is a minimal $k^{\P}$-$\{1,2\}$-factor-critical
planar graph, while $\delta(G)=k^{\P}+1$. That is, both bounds in
\cref{main3} can be attained. The examples obtained motivate the
following conjecture.

\begin{conjecture}\label{asdiouhn12}
	If $G$ is a minimal $t$-$\{1,2\}$-factor-critical
	graph for $t=k,k+1$, then $G$ is a complete graph.
\end{conjecture}

In \cite{KEVINksachsminimal1}, using the ear-pendant decomposition of
$2$-bicritical graphs due to Bourjolly and Pulleyblank
\cite{bourjolly1989konig}, the following result is obtained.

\begin{theorem}[\cite{KEVINksachsminimal1}]
	Let $G$ be a minimal $1$-$\{1,2\}$-factor-critical
	graph different from $K_{4}$. Then $\delta(G)=2$. 
\end{theorem} 

If $G$ is a minimal $2$-$\{1,2\}$-factor-critical graph,
then $\delta(G)\ge3$. Hence, if $G$ is also a minimal
$1$-$\{1,2\}$-factor-critical graph, we obtain that $G=K_{4}$.
This confirms \cref{asdiouhn12} for $k=1$.

\section*{Acknowledgments}

This work was partially supported by Universidad Nacional de San Luis, grants PROICO 03-0723 and PROIPRO 03-2923, MATH AmSud, grant 22-MATH-02, Consejo Nacional de Investigaciones
Cient\'ificas y T\'ecnicas grant PIP 11220220100068CO and Agencia I+D+I grants PICT 2020-00549 and PICT 2020-04064.

\section*{Declaration of generative AI and AI-assisted technologies in the writing process}
During the preparation of this work the authors used ChatGPT-3.5 in order to improve the grammar of several paragraphs of the text. After using this service, the authors reviewed and edited the content as needed and take full responsibility for the content of the publication.

\section*{Data availability}

Data sharing not applicable to this article as no datasets were generated or analyzed during the current study.

\section*{Declarations}

\noindent\textbf{Conflict of interest} \ The authors declare that they have no conflict of interest.

\bibliographystyle{apalike}

\bibliography{TAGcitasV2025}

\end{document}